\documentclass[a4paper]{article}
\usepackage{amsmath,amssymb,graphicx}

\newtheorem{prethm}{{\bf Theorem}}

\newenvironment{thm}{\begin{prethm}{\hspace{-0.5
               em}{\bf.}}}{\end{prethm}}

\newtheorem{prepro}[prethm]{Proposition}
\newenvironment{pro}{\begin{prepro}{\hspace{-0.5
               em}{\bf.}}}{\end{prepro}}

\newtheorem{preconj}[prethm]{Conjecture}
\newenvironment{conj}{\begin{preconj}{\hspace{-0.5
               em}{\bf.}}}{\end{preconj}}

\newtheorem{prelem}[prethm]{Lemma}
\newenvironment{lem}{\begin{prelem}{\hspace{-0.5
               em}{\bf.}}}{\end{prelem}}

\newtheorem{precor}[prethm]{Corollary}
\newenvironment{cor}{\begin{precor}{\hspace{-0.5
               em}{\bf.}}}{\end{precor}}

\newtheorem{prerem}[prethm]{{\bf Remark}}

\newtheorem{preexample}{{\bf Example}}

\newtheorem{preproof}{{\bf Proof.}}

\newenvironment{proof}[1]{\begin{preproof}{\rm
               #1}\hfill{$\Box$}}{\end{preproof}}

\newcommand{\noi}{\noindent}
\newcommand{\ov}{\overline}

\newcommand{\Gb}{{\overline G}}

\renewcommand{\thefootnote}

\title{Nordhaus--Gaddum type inequalities for Laplacian and signless Laplacian eigenvalues}

\author{ F. Ashraf$^{\,\rm 1, 2}$~~~ B. Tayfeh-Rezaie$^{\,\rm 2,}$\thanks{Corresponding author, email: {\tt tayfeh-r@ipm.ir}} \vspace{.4cm}\\
{\sl $^{\rm 1}$Department of Mathematical Sciences, Isfahan University of Technology,}\\
{\sl Isfahan, 84156-83111, Iran}\\
{\sl $^{\rm 2}$School of Mathematics, Institute
for Research in Fundamental Sciences (IPM),} \\
 {\sl  P.O. Box 19395-5746, Tehran, Iran}\\
}
\date{}
\begin{document}
\maketitle

\begin{abstract}
Let $G$ be a graph with $n$ vertices. We denote the largest  signless Laplacian eigenvalue of $G$ by $q_1(G)$ and Laplacian eigenvalues of $G$ by $\mu_1(G)\ge\cdots\ge\mu_{n-1}(G)\ge\mu_n(G)=0$.
It is a conjecture on Laplacian spread of graphs that
$\mu_1(G)-\mu_{n-1}(G)\le n-1$ or equivalently $\mu_1(G)+\mu_1(\Gb)\le2n-1$.
We prove the conjecture for bipartite graphs.
Also we show that for any bipartite graph $G$, $\mu_1(G)\mu_1(\Gb)\le n(n-1)$.
Aouchiche and Hansen [A survey of Nordhaus--Gaddum type relations, Discrete Appl. Math. 161 (2013), 466--546] conjectured that
%for any graph $G$ with $n$ vertices,
$q_1(G)+q_1(\Gb)\le3n-4$ and $q_1(G)q_1(\Gb)\le2n(n-2)$.
We prove the former and disprove the latter by constructing a family of graphs $H_n$ where $q_1(H_n)q_1(\ov{H_n})$ is about $2.15n^2+O(n)$.

\vspace{3mm}
\end{abstract}

\section{Introduction}

Let $G$ be a simple graph with vertex set $V(G) = \{v_1,\ldots, v_n\}$.
We denote the complement graph of $G$ by $\ov G$.
 The {\em degree} of a vertex $v\in V(G)$,
denoted by $d(v)$, is the number of neighbors of $v$.
The {\em adjacency matrix} of $G$ is an $n\times n$ matrix $A(G)$ whose $(i,j)$ entry is 1 if $v_i$ and $v_j$ are adjacent and zero otherwise.
The {\em Laplacian matrix} and the {\em signless Laplacian matrix} of $G$ are the matrices
$L(G) =A(G)-D(G)$ and $Q(G) =A(G)+D(G)$, respectively,
where $D(G)$ is the diagonal matrix with $d(v_1),\ldots, d(v_n)$ on its main diagonal.
It is well-known that $L(G)$ and $Q(G)$ are positive semidefinite
and so their eigenvalues are nonnegative real numbers. The eigenvalues of $L(G)$ and $Q(G)$ are called the
{\em Laplacian eigenvalues} and {\em signless Laplacian eigenvalues} of $G$, respectively, and are denoted by $\mu_1(G)\ge\cdots\ge\mu_n(G)$ and
$q_1(G)\ge\cdots\ge q_n(G)$, respectively. We drop $G$ from the notation when there is no danger of confusion.
Note that each row sum of $L(G)$ is 0 and therefore $\mu_n(G) = 0$. In fact the multiplicity of 0 as an eigenvalue of $L(G)$ equals the number of conncted components of $G$.

Nordhaus and Gaddum \cite{ng} studied the chromatic number in a graph $G$ and in its complement together. They proved lower and upper bounds on the
sum and on the product of chromatic number of $G$ and that of $\ov G$ in terms of the number of vertices $G$. Since then, any bound on the sum and/or the product of an invariant in a graph $G$ and the same invariant in $\ov G$ is called a {\em Nordhaus--Gaddum type inequality}.
In \cite{hans}, Nordhaus--Gaddum type inequalities for graph parameters were surveyed.
Many of those inequalities involve eigenvalues of adjacency, Laplacian and signless Laplacian matrices of graphs.
The first known spectral Nordhaus--Gaddum results belong to Nosal \cite{nos}, and to Amin and Hakimi \cite{amin},
who showed that for every graph $G$ of order $n$,
$$\lambda(G)+ \lambda(\ov G)< \sqrt2(n-1),$$
where $\lambda(H)$ denotes the largest eigenvalues of $A(H)$.
 A minor improvement was obtained by  Nikiforov \cite{niki}.
In the same paper he conjectured that
$$\lambda(G)+ \lambda(\ov G)< \frac{4}{3}n+O(1).$$
The conjecture was proved by Terpai \cite{ter}
who showed that $$\lambda(G)+ \lambda(\ov G)< \frac{4}{3}n-1.$$

In this paper we study Nordhaus--Gaddum type inequalities for Laplacian and signless Laplacian eigenvalues of graphs.
For Laplacian eigenvalues, we have $\mu_1(G)+\mu_1(\Gb)=n+\mu_1(G)-\mu_{n-1}(G)$.
The quantity $\mu_1(G)-\mu_{n-1}(G)$ is called {\em Laplacian spread} of $G$.
It is a conjecture \cite{yl,zsh} that
$\mu_1(G)-\mu_{n-1}(G)\le n-1$ or equivalently $\mu_1(G)+\mu_1(\Gb)\le2n-1$.
We prove the conjecture for bipartite graphs.
Partial results on this conjecture were obtained by several authors \cite{btf,cw,flt,fxwl,lsl,l,lw,xm}.
Also we show that for any bipartite graph $G$, $\mu_1(G)\mu_1(\Gb)\le n(n-1)$.
Aouchiche and Hansen \cite{hans} conjectured that
$q_1(G)+q_1(\Gb)\le3n-4$ and $q_1(G)q_1(\Gb)\le2n(n-2)$.
We prove the former and disprove the latter by constructing a family of graphs $H_n$ where $q_1(H_n)q_1(\ov{H_n})$ is about $2.15n^2+O(n)$.

\section{Preliminaries}

We denote the number of edges of $G$ by $e(G)$.
 We also denote the complete graph on $n$ vertices by $K_n$ and the complete bipartite graph with parts of sizes $r$ and $s$ by $K_{r,s}$.
 The maximum and minimum degrees of $G$ are denoted by $\Delta(G)$ and $\delta(G)$, respectively.
 For a vertex $v\in V(G)$, $m(v)$ denotes the average degree of the neighbors of $v$, i.e. $m(v)=\frac{1}{d(v)}\sum_{u\sim v}d(u)$.
 For any graph parameter $p(G)$ we occasionally use $\ov p$ to denote $p(\ov G)$.

\begin{lem}\label{inter} {\rm(\cite[p. 222]{crsB})} Suppose that $G$ is a graph and $G'$ is
obtained by removing one edge from $G$.
Then the signless Laplacian eigenvalues of
 $G$ and $G'$ interlace:
$$q_1(G)\ge q_1(G')\ge q_2(G)\ge q_2(G')\ge\cdots\ge q_n(G)\ge q_n(G').$$
\end{lem}

\begin{lem}\label{SimBip} {\rm (\cite[p. 217]{crsB})} If $G$ is a bipartite graph, then the matrices $L(G)$ and $Q(G)$ are similar, i.e.
the Laplacian and signless Laplacian eigenvalues of $G$ are the same.
\end{lem}

The following lemma is easy to prove.
\begin{lem}\label{QK_n}
\begin{itemize}
  \item[\rm(i)] The signless Laplacian eigenvalues of $K_n$ are $2n-2$ with multiplicity $1$ and $n-2$ with multiplicity $n-1$.
  The Laplacian eigenvalues of $K_n$ are $n$ with multiplicity $n-1$ and $0$ with multiplicity $1$.
  \item[\rm(ii)] The (signless) Laplacian eigenvalues of $K_{r,s}$ are $r+s$ with multiplicity $1$, $r$ with multiplicity $s-1$, $s$ with multiplicity $r-1$, and $0$ with multiplicity $1$.
\end{itemize}
\end{lem}

\begin{lem}\label{mubar} {\rm (\cite[p. 185]{crsB})}  For any graph $G$, $\mu_i(G)=n-\mu_{n-i}(\Gb)$ for $i=1,\ldots,n-1$.
\end{lem}

From Lemma~\ref{mubar}, one concludes the following.

\begin{lem}\label{mu1=n}   For any graph $G$ with $n$ vertices, $\mu_1(G)\le n$ with equality if and only if $\ov G$ is disconnected.
\end{lem}

\begin{lem} {\rm (Das \cite{das})} \label{das} Let $G$ be a graph with $n$ vertices, $e$ edges, and  $\Delta$, $\delta$ be the largest degree and the
smallest degree of $G$, respectively. Then for any vertex $v$,
\begin{equation}\label{eq:das} d(v)+m(v)\le\frac{2e}{n-1}+\frac{n-2}{n-1}\Delta+(\Delta-\delta)\left(1-\frac{\Delta}{n-1}\right).\end{equation}
If $G$ is connected, then the equality holds if and only if either $d(v) = n - 1$ or $d(v) = \Delta$ and
all neighbors of $v$ have degrees $\Delta$ and all non-neighbors of  $v$ have degrees $\delta$.
\end{lem}

From the Perron--Frobenius Theorem (see \cite[p. 22]{bh}) it follows that the largest eigenvalue of any non-negative square matrix $M$ does not exceed the maximum row-sum of $M$.
Consider the matrix $M=D^{-1}QD=D^{-1}AD+D$. The row-sum of $M$ corresponding with vertex $v$ is $d(v)+m(v)$. Since $q_1(G)$
is the largest eigenvalue of $M$ we have the following.

 \begin{lem} {\rm (Merris \cite{mer})} \label{dv+mv} For any graph $G$,
\begin{equation} \label{eq:dv+mv}q_1(G)\le\max\{d(v)+m(v)\mid v\in V(G)\}.\end{equation}
If $G$ is connected, then the equality holds if and only if $G$ is regular or bipartite semiregular.
\end{lem}

For a graph $G$, consider a partition $P=\{V_1,\ldots,V_m\}$ of $V(G)$.
The partition of $P$ is {\em equitable} if each submatrix $Q_{ij}$ of $Q(G)$ formed by the
rows of $V_i$ and the columns of $V_j$ has constant row sums $r_{ij}$.
The $m\times m$ matrix $R=(r_{ij})$ is called the {\em quotient matrix}
of $Q(G)$ with respect to $P$. The proof of the following theorem is similar to the one given in \cite[p. 187]{crsB} where a similar result is presented for Laplacian matrix.

\begin{lem}\label{equi}
Any eigenvalue of the quotient matrix $R$ is an eigenvalue of $Q(G)$. Moreover,
the largest eigenvalue of $R$ is the largest eigenvalue of $Q(G)$.
\end{lem}

We will use a variant of Lemma~\ref{equi} as follows.
%This is in fact a consequence of Lemma~\ref{equi} and the Perron--Frobenius Theorem.
%We omit the straightforward proof.
We denote the largest eigenvalue of a matrix $M$ by $\lambda_{\max}(M)$.

\begin{lem}\label{equi<}
Let $G$ be graph and $P=\{V_1,\ldots,V_m\}$ be a partition of $V(G)$.
Suppose that each submatrix $Q_{ij}$ of $Q(G)$ formed by the
rows of $V_i$ and the columns of $V_j$ has row sum at most $r_{ij}$ and form
the $m\times m$ matrix $R=(r_{ij})$. Then $q_1(G)\le\lambda_{\max}(R)$.
\end{lem}
\begin{proof}{For all $1\le i,j\le m$, we multiply each row of $Q_{ij}$ by some number (larger than or equal to $1$) so that all the row sums of the resulting matrix $Q'_{ij}$ are equal to $r_{ij}$.
Let $Q'$ be the matrix obtained from $Q(G)$ by replacing the submatrices $Q_{ij}$ by $Q'_{ij}$  for all $1\le i,j\le m$.
Each entry of $Q'$ is larger than or equal to the corresponding entry of $Q$, hence from Perron--Frobenius Theorem it follows that
$\lambda_{\max}(Q)\le\lambda_{\max}(Q')$. On the other hand,  $\{V_1,\ldots,V_m\}$ is an equitable partition for $Q'$ with the qutiont matrix $R=(r_{ij})$.
So, by Lemma~\ref{equi}, $\lambda_{\max}(Q')=\lambda_{\max}(R)$. This completes the proof.
}\end{proof}

\section{Nordhaus--Gaddum type inequalities for Laplacian eigenvalues}

In this section we study Nordhaus--Gaddum type inequalities for both sum and product of Laplacian eigenvalues of a graph.
In view of Lemma~\ref{mubar}, $\mu_1(G)+\mu_1(\ov G)=n+\mu_1(G)-\mu_{n-1}(G)$. %, where $n$ is the order of $G$.
So to study $\mu_1+\ov\mu_1$ it is enough to consider  $\mu_1-\mu_{n-1}$. Some authors have studied this quantity and it is called {\em Laplacian spread}.
Most of the results on Laplacian spread of graphs are around the following conjecture which was posed in \cite{zsh} and appeared as a question in \cite{yl}.

\begin{conj} \label{LaplSpread} For any graph $G$ with $n$ vertices, $\mu_1(G)-\mu_{n-1}(G)\le n-1$ (or equivalently $\mu_1(G)+\mu_1(\Gb)\le 2n-1$)
 and equality holds if and only if $G$ or $\Gb$ is isomorphic to the join of $K_1$ and a disconnected graph of order $n-1$.
\end{conj}
So far, the conjecture has been proved for trees \cite{fxwl}, unicyclic graphs \cite{btf}, bicyclic graphs \cite{flt,lsl,lw}, tricyclic graphs \cite{cw}, cactus graphs \cite{l}, and  quasi-tree graphs \cite{xm}. We remark that Conjecture~\ref{LaplSpread} also holds if $G$ or $\Gb$ is disconnected.
To see this, assuming $\Gb$ is disconnected, by Lemma~\ref{mu1=n}, $\mu_1(G)\le n$ and $\mu_1(\Gb)\le n-1$ and the equality holds in the latter if and only if
$\Gb$ has a connected component of order $n-1$, say $H$, such that $\ov H$ is disconnected.  In other words,
 $\mu_1(\Gb)= n-1$ if and only if
 $G=K_1\vee\ov H$ with $\ov H$ being disconnected  where `$\vee$' shows the join of two graphs.

In this section we prove Conjecture~\ref{LaplSpread} for bipartite graphs (and their complements).
Therefore, our focus in this section will be on bipartite graphs.
By Lemma~\ref{SimBip}, we may consider $Q(G)$ instead of $L(G)$ whenever is required, though  we insist to state our results in terms of Laplacian eigenvalues. %because of the connection with Conjecture~\ref{LaplSpread}.

We start with the following theorem which provides a lower bound on $\mu_{n-1}(G)$ (also called {\em algebraic connectivity} of $G$ \cite{fi}).
For a polynomial $f(x)$ with real zeros, we use the notation $z_{\min}(f)$ and $z_{\max}(f)$ to denote the smallest and largest real zeros of $f(x)$, respectively.

\begin{thm}\label{bip} Let $G$ be a bipartite graph with bipartition $(X,Y)$ and $|Y|\ge|X|$.
If $X$ contains some vertices of degree $|Y|$ and
$Y$ contains $\ell$ vertices of degree $|X|$, then $$\mu_{n-1}(G)\ge\frac{\ell}{|Y|}.$$
Equality holds if and only if $G$ is a star.
\end{thm}
\begin{proof}{Let $k:=|X|\le n/2$ and so $|Y|=n-k$.
Suppose that $X_0$ is the set of vertices in $X$ of degree
$n-k$ with $t:=|X_0|$ and $Y_0$ is the set of vertices in $Y$ of degree $k$, so $|Y_0|=\ell$.
If $X=X_0$ or $Y=Y_0$, then $G$ is the complete bipartite graph $K_{n-k,k}$ and by Lemma~\ref{QK_n}, $\mu_{n-1}(G)=k$.
Thus assume that $1\le t\le k-1$ and $1\le\ell\le n-k-1$.
Let $H$ be the graph obtained from $G$ by removing the edges between $X\setminus X_0$ and $Y\setminus Y_0$.
By Lemma~\ref{inter}, $\mu_{n-1}(G)\ge\mu_{n-1}(H)$. So it suffices to prove the assertion for $H$.
In $H$, any $v\in X\setminus X_0$ has degree $\ell$ and any $v\in Y\setminus Y_0$ has degree $t$.
It turns out that the partition $(X_0,X\setminus X_0,Y_0,Y\setminus Y_0)$ is an equitable partition for $H$.
The corresponding quotient matrix for $Q(H)$ is
$$R=\begin{pmatrix}n-k&0&\ell&n-k-\ell\\0&\ell&\ell&0\\t&k-t&k&0\\t&0&0&t\end{pmatrix}.$$
The characteristic polynomial of $R$ is $xf(x)$ where
\begin{equation}\label{f(x)}
 f(x):=x^3-(n+\ell+t)x^2+(kt+nk+\ell n-\ell k+2\ell t-k^2)x-\ell tn.
\end{equation}
For simplicity, let $r:=n-k$.
Note that in the matrices $Q(H)-rI$, $Q(H)-\ell I$, $Q(H)-kI$ and $Q(H)-tI$, respectively, the rows corresponding to the vertices in
$X_0$,  $X\setminus X_0$, $Y_0$, and $Y\setminus Y_0$ are identical.
So the four matrices above, respectively, have nullities at least $|X_0|-1$, $|X\setminus X_0|-1$, $|Y_0|-1$, and $|Y\setminus Y_0|-1$.
It follows that $Q(H)$ has eigenvalues $r$, $\ell$, $k$, and $t$ with multiplicities at least $t-1$, $k-t-1$, $\ell-1$, and
$r-\ell-1$, respectively. On the other hand, we have
  \begin{align*}
    f(r)&=-2t(r-n/2)(r-\ell),\\
    f(k)&=-2\ell(k-n/2)(k-t),\\
    f(\ell)&=t(r-\ell)(k-t),\\
    f(t)&=\ell(r-\ell)(k-t).
\end{align*}
If $k<n/2$, then none of  $r,\ell,k,t$ is a zero of $f$. It turns out that the polynomial
\begin{equation}\label{fQH}
x(x-r)^{t-1}(x-k)^{\ell-1}(x-\ell)^{k-t-1}(x-t)^{r-\ell-1}f(x)
\end{equation}
is the characteristic polynomial of $Q(H)$.
If $k=n/2$, then $k=r$ is a zero of $f$. From the above argument,
 $x(x-n/2)^{t+\ell-3}(x-\ell)^{n/2-t-1}(x-t)^{n/2-\ell-1}f(x)$
 is a factor of characteristic polynomial of $Q(H)$ and thus
 $n-1$ eigenvalues of $Q(H)$ are determined. Only one eigenvalue is remained which can be determined as the sum of eigenvalues of $Q(H)$ equals $2e(H)$. It turns out that the remaining eigenvalue is also $n/2$.
 So in this case also the characteristic polynomial of $Q(H)$  is \eqref{fQH}.
Note that as $f$ is a cubic polynomial with a positive leading coefficient and as $f(t)>0$, $f(\ell)>0$, and $f(r)\le0$, it follows that
two zeros of $f$ are greater than $\ell$ and $z_{\min}(f)<\min\{\ell,t\}$. Therefore, $\mu_{n-1}(H)=z_{\min}(f)$  and further
 if we show that $f(\ell/r)<0$, then we can conclude that $z_{\min}(f)\ge \ell/r$.
We proceed to show that $f(\ell/r)<0$.

We have
\begin{align*}
\frac{r^3}{\ell}f\left(\frac{\ell}{r}\right)&=-r^4+(n-t+\ell-tn)r^3+(2t\ell+tn)r^2-(t\ell+\ell^2+\ell n)r+\ell^2\\
&=(-\ell r-r^3+2r^2\ell+r^2n-r^3n)t+\ell^2-r^4-\ell^2r-\ell rn+r^3\ell+r^3n.
\end{align*}
It is easy to see that the coefficient of $t$ is always negative, so $f(\ell/r)$ is maximized only if $t=1$.
Therefore,
$$\frac{r^3}{\ell}f\left(\frac{\ell}{r}\right)\le F(r,\ell):=-r^4+(\ell-1)r^3+(n+2\ell)r^2-(\ell+\ell^2+\ell n)r+\ell^2.$$
As $n\le2r$ and $\ell\le r$, we have
\begin{align*}
    \frac{\partial F}{\partial\ell}&=r^3+2r^2-nr-2\ell(r-1)-r\\
    &\ge r^3+2r^2-2r^2-2r(r-1)-r=r(r-1)^2>0.
\end{align*}
So, $F$ is increasing in $\ell$ and thus $F(r,\ell)< F(r,r)=0$.
Therefore, $f(\ell/r)<0$ and so $\mu_{n-1}(G)\ge\mu_{n-1}(H)\ge\ell/r.$

Now we consider the case of equality. For a star obviously equality occurs. Conversely, suppose that $\mu_{n-1}(G)=\ell/|Y|$.
Note that $\ell/|Y|$ is an algebraic integer only if $\ell=|Y|$. Hence $G=K_{k,n-k}$ and thus $k=\mu_{n-1}(G)=\ell/|Y|=1$. This means that $G$ is a star.
}\end{proof}

\begin{thm}\label{mu1} Let $G$ be a bipartite graph with $n$ vertices,  $e$ edges, and bipartition $(X,Y)$. If $|X|\le|Y|$,  then
$$\mu_1(G)\le |Y|+\frac{e}{|Y|}.$$
Moreover, if $Y$ contains $\ell$ vertices of degree $|X|$, then
$$|Y|+\frac{e}{|Y|}\le n-1+\frac{\ell}{|Y|}.$$
The equality $\mu_1(G)=n-1+\ell/|Y|$ holds if and only if $G$ is a complete bipartite graph.
\end{thm}
\begin{proof}{
Let $k:=|X|\le n/2$ and so $|Y|=n-k$.
Let $v\in V(G)$ and $d=d(v)$ and $m(v)=S/d$. Then
\begin{eqnarray}
d+\frac{S}{d}\le(n-k)+\frac{e}{n-k}&\Longleftrightarrow&
    (n-k)d+(n-k)\frac{S}{d}\le(n-k)^2+S+(e-S)\nonumber\\
&\Longleftrightarrow&
    (n-k-d)\left(n-k-\frac{S}{d}\right)+(e-S)\ge0\label{S/d},
\end{eqnarray}
where the last inequality holds as $d$ and $S/d$ are at most $n-k$ and $S\le e$.
 Hence, by Lemma~\ref{dv+mv}, we have
$\mu_1(G)\le n-k+e/(n-k)$.
Since  $Y$ contains $\ell$ vertices of degree $k$, we have $e\le(n-k)(k-1)+\ell$ and so $n-k+e/(n-k)\le n-1+\ell/(n-k)$.

Now we consider the equality case. If $G$ is a complete bipartite graph, then $\mu_1=n$ and the equality holds in both inequalities of the theroem. Conversely, if $\mu_1=n-1+\ell/|Y|$, then  we must have the equality in \eqref{S/d} which is possible only if either $d(v)=n-k$ or $m(v)=S/d=n-k$.
This means that we must have some vertices in $X$ (or in $Y$ if $k=n/2$) of degree $n-k$.
We must also have the equality in \eqref{eq:dv+mv} which implies that all the vertices in $X$ have degree $n-k$, and thus
 $G$ is $K_{n-k,k}$.
}\end{proof}

Combining Theorems~\ref{bip} and \ref{mu1}, we conclude the following. We note that Theorem~\ref{bip} may not hold in the case that $X_0=\emptyset$ in which case we have $\mu_1(G)<n-1$ by Part 4 of Theorem~\ref{t&l}.
\begin{thm} If $G$ or $\Gb$  is a bipartite graph with $n$ vertices, then the following equivalent inequalities hold:
\begin{itemize}
  \item $\mu_1(G)+\mu_1(\Gb)\leq 2n-1$;
  \item $\mu_1(G)-\mu_{n-1}(G)\leq n-1$;
  \item $\mu_{n-1}(G)+\mu_{n-1}(\Gb)\geq1$.
\end{itemize}
Equality holds if and only if $G$ or $\Gb$ is a star.
\end{thm}

 We now proceed to establish a Nordhaus--Gaddum type inequality for the product of Laplacian eigenvalues.
 More precisely we will show that $\mu_1\ov\mu_1\leq n(n-1)$ for bipartite graphs.
 However, this cannot be concluded from Theorems~\ref{bip} and \ref{mu1}.
 We need to obtain sharper bounds on the eigenvalues. This will be done in the next theorem.

\begin{thm}\label{t&l} Let $G$ be a bipartite graph with bipartition $(X,Y)$ and $|Y|\ge|X|\ge2$. Suppose that $X$ contains $t$ vertices of degree $|Y|$ and
$Y$ contains $\ell$ vertices of degree $|X|$.
\begin{enumerate}
  \item If $t\ge2$ and $\ell\ge2$, then either $\mu_{n-1}>1$ or $\mu_1\le n-1$.
  \item For $\ell=1$:  if $t\le k-2$, then $\mu_1<n-1$; if $t=k-1<n/2-1$, then $\mu_1< n-1+1/n$; if $t=k-1=n/2-1$, then $\mu_1=\frac{1}{2}\left(n+\sqrt{n^2-4n+8}\right)$ and $\mu_{n-1}=\frac{1}{2}\left(n-\sqrt{n^2-4n+8}\right)$.

  \item If $t=1$ and $n\ge7$, then $\mu_1< n-1+\ell/n$.
   \item If either $t=0$ or $\ell=0$, then $\mu_1< n-1$.
\end{enumerate}
\end{thm}
\begin{proof}{
Let  $k:=|X|\le n/2$ and $|Y|=n-k$.
If $t=0$, i.e. all the vertices in $X$ have degree at most $n-k-1$, then by applying Lemma~\ref{equi<}
for bipartition of $G$, we have
$$\mu_1<\lambda_{\max}
\begin{pmatrix}
 n-k-1 & n-k-1 \\  k & k \end{pmatrix}=n-1.$$
Similarly, if $\ell=0$, then $\mu_1<n-1$. This implies Case 4.
In the rest of the proof we assume that  $t\ge1$ and $\ell\ge1$.

\noi{\bf Case 1.} $t\ge2$ and $\ell\ge2$.

Suppose that $X_0$ is the set of vertices in $X$ of degree
$n-k$ and $Y_0$ is the set of vertices in $Y$ of degree $k$.
If we apply Lemma~\ref{equi<} on $Q(G)$
with the partition $(X_0,X\setminus X_0,Y_0,Y\setminus Y_0)$,
it turns out that $\mu_1$ does not exceed the largest eigenvalue of the matrix
$$R=\begin{pmatrix}n-k&0&\ell&n-k-\ell\\0&n-k-1&\ell&n-k-\ell-1\\t&k-t&k&0\\t&k-t-1&0&k-1\end{pmatrix}.$$
The characteristic polynomial of $R$ is $xg(x)$ where
$$g(x)=x^3+(2-2n)x^2+(n^2+nk-2n-k^2-\ell-t)x+\ell n+nk^2+2nk-\ell k+tk-n^2k-2k^2.$$
  We have
$$g(n-1)=(k-t-1)n-k^2+\ell+t+tk-\ell k+1.$$
On the other hand, from the proof of Theorem~\ref{bip} we know that $\mu_{n-1}\ge z_{\min}(f)$ with $f(x)$  given in \eqref{f(x)}.
Also, if we have $f(1)<0$, then  $z_{\min}(f)>1$.
Since $g(n-k)=t(2k-n)\le0$ and $g(k)=\ell(n-2k)\ge0$, two zeros of $g$ are at most $n-k$.  It follows that if $g(n-1)>0$, then $z_{\max}(g)< n-1$.
Next, note that
$$f(1)=(\ell+k-t\ell-1)n-t-\ell-\ell k+kt+2t\ell-k^2+1,$$
and thus
$$g(n-1)-f(1)=(n-2)(t\ell-t-\ell).$$
By the assumptions $t\ge2$ and $\ell\ge2$, we have $t\ell-t-\ell\ge0$.
Hence $g(n-1)-f(1)\ge0$.
Now, if $\mu_{n-1}\le1$, then $z_{\min}(f)\le1$, hence $f(1)\ge0$.
From the above equation it follows that $g(n-1)\ge0$, and so $\mu_1\le z_{\max}(g)\le n-1$.
Therefore, we have either $\mu_{n-1}>1$ or $\mu_1\le n-1$.

\noi{\bf Case 2.} $\ell=1$.

In this case we have $g(n-1)=(k-t)(n-1-k)+2-n$. It follows that $g(n-1)>0$ provided that $k-t\ge2$.
Hence, if $t\le k-2$, then $\mu_1<n-1$, and we are done. So, assume that $t=k-1$.
It turns out that in this subcase, there are no edges between $X\setminus X_0$ and $Y\setminus Y_0$ and hence $G$ is isomorphic to the graph $H$ of Theorem~\ref{bip}. So $\mu_1=z_{\max}(f)$.
First assume that $k<n/2$. We have,
$$f\left(n-1+\frac{1}{n}\right)=\frac{1}{n^3}\Big(n^4-2(k+1)n^3+(2k+3)n^2-(k+3)n+1\Big),$$
which is minimized when $k$ is maximized, i.e. $k=(n-1)/2$. Thus $f(n-1+1/n)\ge (3n^2/2-5n/2+1)/n^3$ which is positive for $n\ge2$.
Therefore, $\mu_1<n-1+1/n$.
If $k=n/2$, then the three zeros of $f$ are $\frac{n}{2},\frac{1}{2}\left(n\pm\sqrt{n^2-4n+8}\right)$.

\noi{\bf Case 3.} $t=1$.

In this case we show that $z_{\max}(g)< n-1+\ell/n$. It suffices to show that $g(n-1+\ell/n)>0$.
We have $g(n-1+\ell/n)=n^{-3}h(k,\ell)$ where
$$h(k,\ell)=(k-2)n^4+(k-k^2+\ell+2)n^3-\ell(k^2+2)n^2-n\ell^2+\ell^3.$$
 By computation, for $n\ge7$, we have
\begin{align*}
    h(2,\ell)&=\ell (n^3-6n^2-n\ell+\ell),\\
             &\ge\ell (n^3-6n^2-n^2+\ell)>0.
\end{align*}
So we may assume that $k\ge3$.
We may write $h(k,\ell)=h_1(\ell)+n^2h_2(k,\ell)$ where
\begin{align*}
    h_1(\ell)&=2n^3-2\ell n^2-n\ell^2+\ell^3,\\
    h_2(k,\ell) &=(k-2)n^2+(k-k^2)n+\ell(n-k^2).
\end{align*}
The function $h_1$ is decreasing in $\ell$, and so $h_1(\ell)> h_1(n)=0$.
Now for showing the positivity of $h_2$ we consider two subcases: $n\ge k^2$ and $n<k^2$.
If $n\ge k^2$, then
$$h_2(k,\ell)>(k-2)n^2-(k-1)nk$$
which is positive for $k\ge3$.
Next suppose that  $n<k^2$. As $\ell\le n-k-1$ (for if $\ell=n-k$, then $t=k\ge2$), we have
\begin{align}
    h_2(k,\ell) &\ge(k-2)n^2+(k-k^2)n+(n-k-1)(n-k^2)\nonumber\\
     &=(k-1)n^2-(2k^2+1)n+k^2(k+1).\label{h2(k,l)}
\end{align}
If $k=3$, then the right side of \eqref{h2(k,l)} becomes $2n^2-19n+36$ which is positive for $n\ge7$.
The zeros of the quadratic form $(k-1)x^2-(2k^2+1)x+k^2(k+1)$ are
$(2k^2+1\pm\sqrt{8k^2+1})/(2k-2)$, and for $k\ge4$ we have $$\frac{2k^2+1+\sqrt{8k^2+1}}{2k-2}<2k.$$
Since $n\ge2k$, it follows that $h_2(k,\ell)\ge0$.

Therefore, $g(n-1+\ell/n)>0$ and thus
$\mu_1\le z_{\max}(g)< n-1+\ell/n$.
}\end{proof}

%and
%\begin{align*}
%    p(n/2,\ell)&=\frac{1}{4}n^5-\frac{6+\ell}{4}n^4+(\ell+2)n^3-2\ell n^2-n\ell^2+\ell^3\\
%    &\ge p(n/2,1)=\frac{1}{4}n^5-\frac{7}{4}n^4+3n^3-2n^2-n+1>0.
%\end{align*}

%So $p(n/2,\ell)-p(2,\ell)=n^5/4-7n^4/4+2n^3+4n^2>0$. It follows that
%$p(k,\ell)\ge p(2,\ell)$ for any $\ell$.
%It easily seen that $p(2,\ell)>0$ for any $\ell$ and $n\ge7$.
%Now considering $g(t,\ell)=(k^2-nk+1-n+l)t-k^3+nk^2-nk+k^2$ as a function of $t$ and $\ell$, it is seen that $g$ is increasing with
%respect to $\ell$ and decreasing with respect to $t$. It follows that if $t\le k-3$, then
%$g(t,\ell)\ge g(k-3,1)=nk-2k^2+3n+2k-6$ which is positive for $n\ge2$. This implies $\mu_1<n-1$ for $t\le k-3$.

\begin{thm} Let $G$ be a bipartite graph with $n$ vertices. Then $\mu_1(G)\mu_1(\Gb)\leq n(n-1)$, equality holds if and only if either $G$ or $\Gb$ is a star.
\end{thm}
\begin{proof}{As discussed in the paragraph following Conjecture~\ref{LaplSpread}, $\mu_1(G)\le n$ and $\mu_1(\Gb)\le n-1$ with equality if and only if $G$ is a star. Hence, we may assume that both $G$ and $\Gb$ are connected.
We use the notation of Theorem~\ref{t&l}.
If $k=1$, then $G$ is a star and we are done.
So assume that $k\ge2$. The connectedness of $\Gb$ also implies that $t\le k-1$ and $\ell\le n-k-1$.
The theorem for bipartite graphs under the above conditions with $n\le6$ can be easily verified.
So we may assume that $n\ge7$.

By Lemma~\ref{mubar}, we have $\ov\mu_1=n-\mu_{n-1}$.
Therefore, in view of Theorem~\ref{bip},
\begin{equation}\label{mu1bar}
   \ov\mu_1<n-\frac{\ell}{n-k}.
\end{equation}

If one of the cases $t=\ell=0$, or $t\ge2,\ell\ge2$, or $\ell=1,t\le k-2$ occurs, then by Theorem~\ref{t&l} we have either $\mu_1\le n-1$ or $\ov\mu_1<n-1$,
and we are done.
If $\ell=1$ and $t=k-1<n/2-1$, then by Theorem~\ref{t&l} and \eqref{mu1bar},
$$\mu_1\ov\mu_1<\left(n-1+\frac{1}{n}\right)\left(n-\frac{1}{n-k}\right)<n(n-1).$$
If $\ell=1$ and $t=k-1=n/2-1$, then by Theorem~\ref{t&l},
$$\mu_1\ov\mu_1=\frac{1}{4}\left(n+\sqrt{n^2-4n+8}\right)^2$$
which is smaller than $n(n-1)$ for $n\ge4$.
If $t=1$, the result follows similarly from Theorem~\ref{t&l} and \eqref{mu1bar}.
}\end{proof}

Based on the above results, we pose the following conjecture.

\begin{conj}  For any graph $G$ with $n$ vertices, $\mu_1(G)\mu_1(\Gb)\le n(n-1)$
 and  equality holds if and only if $G$ or $\Gb$ is isomorphic to the join of $K_1$ and a disconnected graph of order $n-1$.
\end{conj}

\section{Nordhaus--Gaddum type inequalities for signless Laplacian eigenvalues}

In this section we study Nordhaus--Gaddum type inequalities for signless Laplacian eigenvalues of a graph.
 Aouchiche and Hansen \cite{hans} surveyed Nordhaus--Gaddum type inequalities for graph parameters. Among other things, they gave the following two conjectures which is the subject of this section.

\begin{conj} {\rm(Aouchiche and Hansen \cite{hans})}\label{q+qb} Let $G$ be a simple graph on $n\geq 2$ vertices. Then
$q_1(G)+q_1(\overline{G})\leq3n-4$.
Equality holds if and only if $G$ is the star $K_{1,n-1}$.
 \end{conj}
\begin{conj} {\rm(Aouchiche and Hansen \cite{hans})}\label{q.qb} Let $G$ be a simple graph on $n\geq 2$ vertices. Then
$q_1(G)\cdot q_1(\overline{G})\leq2n(n-2)$.
Equality holds  if and only if $G$ is the star $K_{1,n-1}$.
 \end{conj}

We prove Conjecture~\ref{q+qb} by establishing a more general result and disprove Conjecture~\ref{q.qb}.

\begin{thm}\label{q1+q1bar} Let $G$ be a simple graph on $n\geq 2$ vertices. Then
\begin{equation}\label{eq:q1+q1bar}
q_1(G)+q_1(\overline{G})\leq2n-2+(\Delta-\delta)\left(2-\frac{\Delta-\delta+1}{n-1}\right).
\end{equation}
Equality holds if and only if $G$ is either regular or the star $K_{1,n-1}$.
\end{thm}
\begin{proof}{By Lemmas~\ref{das} and \ref{dv+mv} we have
$$q_1+\overline q_1\le\frac{2e+2\ov e}{n-1}+\frac{n-2}{n-1}(\Delta+\ov\Delta)+(\Delta-\delta)\left(1-\frac{\Delta}{n-1}\right)+(\ov\Delta-\ov\delta)\left(1-\frac{\ov\Delta}{n-1}\right).
$$
By the fact that $e+\ov e=n(n-1)/2$, $\ov\Delta=n-1-\delta$ and $\ov\delta=n-1-\Delta$, we obtain
$$q_1+\overline q_1\le n+\frac{n-2}{n-1}(\Delta-\delta+n-1)+(\Delta-\delta)\left(1-\frac{\Delta}{n-1}+\frac{\delta}{n-1}\right).$$
This implies \eqref{eq:q1+q1bar}.

Now we consider, the  equality case in \eqref{eq:q1+q1bar}.
If $G$ is regular both side of \eqref{eq:q1+q1bar} are equal $2n-2$. If $G=K_{1,n-1}$, both sides of \eqref{eq:q1+q1bar} are equal $3n-4$.
Now, assume that equality occurs in \eqref{eq:q1+q1bar} for $G$.
Then the equalities must hold in \eqref{eq:das} and \eqref{eq:dv+mv} for both $G$ and $\ov G$.
We may assume that $G$ is not regular. By the case of equality in \eqref{eq:dv+mv}, $G$ must be a bipartite semiregular for which the equality does not occur in  \eqref{eq:das} unless $\Delta=n-1$. Hence $G=K_{1,n-1}$.
}\end{proof}

\begin{cor} Conjecture~$\ref{q+qb}$ holds.
\end{cor}
\begin{proof}{
By Theorem~\ref{q1+q1bar}, it suffices to show that $$(\Delta-\delta)\left(2-\frac{\Delta-\delta+1}{n-1}\right)\le n-2.$$
Note that for any graph $G$, $\Delta-\delta\le n-2$. Let
$$f(x):=x\left(2-\frac{x+1}{n-1}\right).$$
Then $f'(x)=2-(2x+1)/(n-1)$ which is positive for $x<n-3/2$.
It follows that $$f(\Delta-\delta)\le f(n-2)=n-2,$$
as desired.
The equality holds if and only if $\Delta-\delta=n-2$. From the equality case in Theorem~\ref{q1+q1bar}, this is possible if and only if $G=K_{1,n-1}$.
}\end{proof}

Noting that $\frac{5}{18}\left(4+\sqrt{14}\right)\approx2.15$, the following proposition disproves Conjecture~\ref{q.qb}.
\begin{pro} For any positive integer $n$, there is a graph $H_n$ on $n$ vertices with
 $$q_1(H_n)q_1(\ov{H_n})=\frac{5}{18}\left(4+\sqrt{14}\right)n^2+O(n).$$
\end{pro}
\begin{proof}{
Let $n=6k+s$ for some positive integers $k$ and $-3\le s\le2$. Set $H_n=\ov{K_{n-k}}\vee K_k$, where `$\vee$' shows the join of two graphs.
Then, $\ov{H_n}=K_{n-k}\cup\ov{K_k}$ with $q_1(\ov{H_n})=2(n-k-1)$. The partition of $V(H_n)$ into the $k$-clique and the $(n-k)$-independent set is an equitable partition with the quotient matrix
$$\left(\begin{array}{cc} n+k-2& n-k \\ k& k \\\end{array}\right).$$
It follows that $$q_1(H_n)=\frac{n}{2}+k-1+\frac{1}{2}\sqrt{n^2+4nk-4n-4k^2+4}.$$
Substituting $k=(n-s)/6$ yields that
\begin{align*}
q_1(H_n)q_1(\ov{H_n})&=\frac{5n+s-6}{18}\left(4n-s-6+\sqrt{14n^2-4ns-36n-s^2+36}\right)\\
&=\frac{5}{18}\left(4+\sqrt{14}\right)n^2+O(n).
\end{align*}
}\end{proof}
Based on the above result, we pose the following.

\noindent{\bf Problem.} Let $\rho(n):=\max\{q_1(G)q_1(\ov G)\mid\hbox{$G$ is a simple graph of order $n$}\}$. Is it true that $\displaystyle\lim_{n\to\infty}\frac{\rho(n)}{n^2}$ exists and equals $\frac{5}{18}\left(4+\sqrt{14}\right)$?


\begin{thebibliography}{MM}
\bibitem{amin} A.T. Amin and S.L. Hakimi, Upper bounds on the order of a clique of a graph, {\em SIAM J. Appl. Math.} {\bf22} (1972), 569--573.
\bibitem{hans} M. Aouchiche and P. Hansen, A survey of Nordhaus--Gaddum type relations, {\em Discrete Appl. Math.} {\bf161}  (2013), 466--546.
\bibitem{btf} Y.-H. Bao, Y.-Y. Tan, and Y.-Z. Fan, The Laplacian spread of unicyclic graphs, {\em Appl. Math. Lett.} {\bf22} (2009),  1011--1015.
\bibitem{bh} A.E. Brouwer and W.H. Haemers, {\em Spectra of Graphs}, Springer, New York, 2012.
\bibitem{cw} Y.  Chen and L. Wang, The Laplacian spread of tricyclic graphs, {\em Electron. J. Combin.} {\bf16} (2009), Research Paper 80, 18 pp.
%\bibitem{csik} P. Csikv\'ari, On a conjecture of V. Nikiforov, {\em Discrete Math.} {\bf309} (2009), 4522--4526.
\bibitem{crsB} D.M. Cvetkovi\'c, P. Rowlinson, and S.K. Simi\'c, {\em An Introduction to the Theory of Graph Spectra}, Cambridge University Press, Cambridge, 2010.
%\bibitem{ceve} D.M. Cvetkovi\'c and S.K. Simi\'c, Towards a spectral theory of graphs based on the signless laplacian, III, {\em Appl. Anal. Discrete Math.} {\bf4} (2010), 156--166.
\bibitem{das} K.Ch. Das, Maximizing the sum of squares of the degrees of a graph, {\em Discrete Math.} {\bf285} (2004), 57--66.
\bibitem{flt} Y.Z. Fan, S.D. Li, and Y.Y. Tan, The Laplacian spread of bicyclic graphs, {\em J. Math. Res. Exposition} {\bf30} (2010), 17--28.
\bibitem{fxwl} Y.-Z. Fan, J. Xu, Y. Wang, and D. Liang, The Laplacian spread of a tree, {\em Discrete Math. Theor. Comput. Sci.} {\bf10} (2008),  79--86.
\bibitem{fi} M. Fiedler, Algebraic connectivity of graphs, {\em Czechoslovak Math. J.} {\bf23(98)} (1973), 298--305.
\bibitem{lsl} P. Li, J.S. Shi, and R.L. Li, Laplacian spread of bicyclic graphs, (Chinese) {\em J. East China Norm. Univ. Natur. Sci. Ed.} {\bf2010}, 6--9.
\bibitem{l} Y. Liu, The Laplacian spread of cactuses, {\em Discrete Math. Theor. Comput. Sci.} {\bf12} (2010), 35--40.
\bibitem{lw} Y. Liu and L. Wang, The Laplacian spread of bicyclic graphs, {\em Advances in Mathematics\,(China)} {\bf40} (2011), 759--764.
\bibitem{mer}  R. Merris, A note on Laplacian graph eigenvalues, {\em Linear Algebra Appl.} {\bf285} (1988), 33--35.
\bibitem{niki} V. Nikiforov, Eigenvalue problems of Nordhaus--Gaddum type, {\em Discrete Math.} {\bf307} (2007), 774--780.
\bibitem{ng} E.A. Nordhaus and J. Gaddum, On complementary graphs, {\em Amer. Math. Monthly} {\bf63} (1956), 175--177.
\bibitem{nos} E. Nosal, {\em Eigenvalues of Graphs}, Master's thesis, University of Calgary, 1970.
\bibitem{ter} T. Terpai, Proof of a conjecture of V. Nikiforov, {\em Combinatorica} {\bf31} (2011), 739--754.
\bibitem{xm} Y. Xu and J. Meng, The Laplacian spread of quasi-tree graphs, {\em Linear Algebra Appl.} {\bf435} (2011), 60–66.
\bibitem{yl} Z. You and B. Liu, The Laplacian spread of graphs, {\em Czechoslovak Math. J.} {\bf62(137)} (2012),  155--168.
\bibitem{zsh} M. Zhai, J. Shu, and Y. Hong, On the Laplacian spread of graphs, {\em Appl. Math. Lett.} {\bf24} (2011), 2097--2101.



\end{thebibliography}
\end{document}